\documentclass[a4paper, 12pt]{amsart}

\usepackage{graphicx}
\usepackage{amsmath}
\usepackage{amsfonts}
\usepackage{amsthm}
\usepackage{amssymb}
\usepackage[table]{xcolor}
\usepackage{stmaryrd}
\usepackage{bm}
\usepackage{tikz-cd}
\usepackage{enumitem}
\usepackage{soul}
\usepackage{mathtools}

\usepackage[utf8]{inputenc} 
\usepackage[T1]{fontenc}             
\usepackage{lmodern}
\usepackage[numbers]{natbib}
\usepackage{framed}
\usepackage{stackrel}
\usepackage{array}
\usepackage{mathabx}

\usepackage{xifthen}% provides \isempty test

\usepackage[margin=4cm]{geometry}
\usepackage{booktabs}
\usepackage{longtable}

\setlist[enumerate,1]{label={(\alph*)}}

\newcommand{\calO}{\mathcal{O}}

 \newcommand{\R}{\mathbb{R}}
 
 \newcommand{\Q}{\mathbb{Q}}
 \newcommand{\Z}{\mathbb{Z}}

\newcommand{\longto}{\longrightarrow}
 \newcommand{\inj}{\hookrightarrow}

 \newcommand{\surj}{\twoheadrightarrow}

\let\oldforall\forall
\renewcommand{\forall}{\; \oldforall}
\let\oldexists\exists
\renewcommand{\exists}{\; \oldexists}

\newcommand{\la}{\langle}
\newcommand{\ra}{\rangle}

\newcommand{\bs}{\backslash}

\DeclareMathOperator{\Ad}{Ad}
\DeclareMathOperator{\tr}{tr}

\theoremstyle{plain} \newtheorem*{theorem*}{Theorem}
\theoremstyle{plain} \newtheorem*{conjecture*}{Conjecture}
\theoremstyle{plain} \newtheorem*{lemma*}{Lemma}
\theoremstyle{plain} \newtheorem*{corollary*}{Corollary}
\theoremstyle{plain} \newtheorem*{proposition*}{Proposition}

\theoremstyle{plain} \newtheorem{thm}{Theorem}[section]
\theoremstyle{plain} \newtheorem{theorem}[thm]{Theorem}
\theoremstyle{plain} 
\theoremstyle{plain} 
\theoremstyle{plain} 
\theoremstyle{definition} 
\theoremstyle{definition} 

\theoremstyle{plain} 
\theoremstyle{plain} 
\theoremstyle{plain} 
\theoremstyle{plain} 
\theoremstyle{plain} 
\theoremstyle{definition} 
\theoremstyle{definition} 
\theoremstyle{plain} 

\theoremstyle{plain} 
\theoremstyle{plain} 
\theoremstyle{plain} \newtheorem*{claim*}{Claim}
\theoremstyle{definition} \newtheorem{remark}[thm]{Remark}
\theoremstyle{definition} 

\theoremstyle{plain}

\theoremstyle{plain} \newtheorem*{theoreme*}{Theorème}
\usepackage{booktabs}
\usepackage{longtable}

\theoremstyle{plain} 
\theoremstyle{plain} 
\theoremstyle{plain} \newtheorem*{lemme*}{Lemme}
\theoremstyle{plain} 
\theoremstyle{plain} \newtheorem*{corollaire*}{corollaire}
\theoremstyle{definition} 

\theoremstyle{plain} 
\theoremstyle{definition} 
\theoremstyle{definition}

\newcommand{\Hy}{\mathbf{H}}

\newcommand{\G}{\mathbf{G}}
\newcommand{\aPO}{\mathbf{PO}}

\renewcommand{\>}{\rangle}

\DeclareMathOperator{\Comm}{Comm}

\DeclareMathOperator{\Isom}{Isom}

\makeatletter
\renewcommand\subsection{\@startsection{subsection}{1}%
  \z@{.5\linespacing\@plus.7\linespacing}{-.5em}%
  {\normalfont}}
\def\@seccntformat#1{%
  \protect\textup{\protect\@secnumfont
    \ifnum\pdfstrcmp{subsection}{#1}=0 \bfseries\fi% subsection # in \bfseries
    \csname the#1\endcsname
    \protect\@secnumpunct
  }%
}  
\makeatother

\title[Non-commensurable manifolds with the same trace ring]%
{Non-Commensurable hyperbolic manifolds with the same trace ring}

\author{Olivier Mila}
\address{Centre de recherches mathématiques,
	Université de Montréal,\linebreak
  Pavillon André-Aisenstadt,
	Montréal, Québec, H3T~1J4, Canada}
  \email{olivier.mila@umontreal.ca, \textit{\fontfamily{\familydefault}\selectfont Web: }crm.umontreal.ca/\textasciitilde mila}

\begin{document}
\begin{abstract}
We prove that there are infinitely many pairwise non-commensurable hyperbolic $n$-manifolds that have the same ambient group and trace ring, for any $n \geq 3$.
The manifolds can be chosen compact if $n \geq 4$.
\end{abstract}
\thanks{This work is supported by the Swiss National Science Foundation, Project number \texttt{P2BEP2\_188144}}
\maketitle
\section*{} \label{sec:intro}

In the study of hyperbolic manifolds of dimension $n \geq 3$, the nicest family is (arguably) the family 
of arithmetic manifolds.
At the core of their definition are two objects of algebraic nature: a field $K$ and an algebraic group $\G$.
In simple terms, an arithmetic lattice is then essentially just $\G(\calO_K)$, and an 
arithmetic manifold the quotient of hyperbolic space by such a lattice.

In his 1971 article \cite{Vinberg}, Vinberg introduced similar objects for arbitrary Zariski-dense subgroups
(in particular, lattices) in semisimple Lie groups.
He defined the (adjoint) trace field (and ring) and the ambient group of such subgroups, and 
proved that they are invariant under commensurability.
These are the main algebraic invariants used to study hyperbolic $n$-manifolds of arbitrary dimension; for 
$n \geq 4$ see for instance~\cite{Emery-Mila}, and for $n = 3$ (and $n=2$) essentially equivalent invariants 
are thoroughly studied in \cite{MR}.
It is also an important tool when one is interested in proving nonarithmeticity and non-commensurability 
(see \cite{Mila18,Mila20}).

In the arithmetic case, these invariants completely determine the commensurability class
\cite[Prop.~2.5]{PR}.
In fact, even more is true: if a manifold $M_1$ has the same ambient group and trace ring as an 
arithmetic manifold $M_2$, then $M_1$ is also arithmetic and commensurable with $M_2$.
It is hence tempting to hope that it still holds in the general case, i.e., that the pair 
(ambient group, trace ring) determines the commensurability class of an arbitrary hyperbolic manifold.

Alas, we prove that this does not hold in general: Theorems~\ref{th:non-compact} and \ref{th:compact} below 
establish 
the existence of infinitely many pairwise non-commensurable hyperbolic manifolds with the same ambient group 
and trace ring, in the non-compact case for $n \geq 3$ and in the compact case for $n \geq 4$, respectively.
Their proof is based on the (now classical) construction of Agol-Belolipetsky-Thomson \cite{Agol, BT} of 
manifolds with short systole, and on the analysis of their trace ring carried out by the author in 
\cite{Mila18}.
Observe that the case $n = 2$ was already known (see \cite[Ex.~4.9.3]{MR}). 
%if not known, not surprising (due to the abundance of examples coming from Dehn surgery).

The article is structured as follows: in Section~\ref{sec:background} we introduce the necessary background,
Section~\ref{sec:non-compact} is devoted to the non-compact case and Section~\ref{sec:compact} to the 
compact case.

\section{Background}\label{sec:background}

\subsection{} 
In this paper all manifolds are assumed to be hyperbolic, complete without boundary and of finite volume.
Equivalently, a manifold $M$ is a quotient $M = \Gamma \bs \Hy^n$ where $\Hy^n$ denotes 
hyperbolic $n$-space and $\Gamma$ is a torsion-free lattice in the Lie group $G =\Isom(\Hy^n)$.
We will use the so-called ``$f$-hyperboloid'' models for hyperbolic space,
defined as follows:
For a real quadratic form $f$ of signature $(n,1)$, let 
\[
    \Hy_f = \{x \in \R^{n+1} \mid f(x) = -1\} / \{\pm 1\},
\]
with the Riemannian structure induced by setting $T_x\Hy_f = x^{\perp_f}$, the $f$\nobreakdash-orthogonal complement of $x$.
In this model, the isometry group \linebreak $\Isom(\Hy_f)$ is identified with $\aPO_f(\R)$, the real points of 
the algebraic group $\aPO_f$.
Observe that $\aPO_f$ is defined over the subfield $K\subset \R$ whenever $f$ is.

\subsection{} 
Following Vinberg \cite{Vinberg}, we define the \emph{(adjoint) trace field} of a \linebreak Zariski-dense 
subgroup $\Gamma \subset G$ 
as the field $K = \Q(\tr \Ad(\gamma) \mid \gamma \in \Gamma)$, where $\Ad$ denotes the adjoint representation.
If $\Gamma$ is a lattice, $K$ is a number field \cite[Chap.~1, §6]{Vinberg-Lie-Groups}.

The \emph{(adjoint) trace ring} of $\Gamma$ is the ring $A$ defined as the integral closure of %the ring 
$\Z[\tr \Ad(\gamma)\mid \gamma \in \Gamma]$.
When $K$ is a number field (in particular, when $\Gamma$ is a lattice), we simply have 
\[
    A = \calO_K[\tr \Ad(\gamma) \mid \gamma \in \Gamma],
\]  
where $\calO_K$ denotes the ring of integers of $K$.
Both the trace field and the trace ring are invariant under commensurability \cite[Th.~3]{Vinberg} (recall that two subgroups of a group are \emph{commensurable} if they share a 
finite-index subgroup, up to conjugation).

Finally the \emph{ambient group} of $\Gamma$ is any algebraic group $\G$ defined over $K$ such that 
$\G(\R) \cong G = \Isom(\Hy^n)$ and $\Gamma \subset \G(K)$ via this isomorphism.
Such a group always exists (take the Zariski-closure of the image of $\Gamma$ under the representation given 
by \cite[Th.~1]{Vinberg}).
Furthermore, by Zariski-density of $\Gamma$ it is unique up to $K$-isomorphism, and up to commensurability
we actually have $\Gamma \subset \G(A)$.

\subsection{}
With this machinery, it is easy to define arithmetic lattices.
A lattice $\Gamma \subset G = \Isom(\Hy^n)$ is \emph{arithmetic} if:
\begin{enumerate}
    \item its trace field $K$ is a totally real number field,
    \item its trace ring is $\calO_K$,
    \item its ambient group $\G$ is \emph{admissible}, meaning that $\G(k \otimes_\Q \R)$ is isomorphic 
        (as a Lie group) to $G \times K$ with $K$ compact.
\end{enumerate}
Recall that here $\G(k \otimes_\Q \R) = \prod_\sigma {}^\sigma \G (\R)$, where the product is over all 
embeddings $\sigma\:K \inj \R$, and ${}^\sigma \G$ denotes the $\sigma$-conjugate group of $\G$.

In this article, the only arithmetic lattices we will consider are those of the \emph{simplest type}, meaning that 
their ambient group $\G$ is isomorphic to $\aPO_f$ for some quadratic form $f$ defined over the trace field $K$.
For such a lattice $\Gamma$, the \emph{arithmetic manifold} $M = \Gamma \bs \Hy_f$ will be non-compact 
if and only if $f$ is isotropic (see \cite[Th.~11.6]{BHC}).
For $n \geq 4$ this forces that (and, for the simplest type, is actually equivalent to) $K = \Q$.

\subsection{}
Let $f$ be a signature $(n,1)$ quadratic form.
A \emph{hyperplane} of $\Hy_f$ is a codimension one totally geodesic subspace.
Equivalently, it is the image in $\Hy_f$ of the $f$-orthogonal complement $v^{\perp_f}$ of a vector $v$ with 
$f(v) >0$.
Two hyperplanes $R_1 = v_1^{\perp_f}$ and $R_2 = v_2^{\perp_f}$ are either \emph{incident} if they meet in 
$\Hy_f$, \emph{asymptotically parallel} if they meet at infinity, or \emph{ultraparallel} otherwise. 
These correspond to the cases when the square of the $f$-scalar product $\la v_1, v_2\ra_f ^2$ is greater than, equal to, or smaller than 
$f(v_1)f(v_2)$ respectively, see \cite[Th.~3.2.7]{Ratcliffe}.
In the ultraparallel case, there is a unique shortest geodesic between $R_1$ and $R_2$ and orthogonal to both.
Its length is the \emph{(hyperbolic) distance} $d(R_1, R_2)$ between $R_1$ and $R_2$, and we have 
(see \cite[Th.~3.2.8]{Ratcliffe})
\[
    \cosh d(R_1,R_2)= \frac{|\la v_1, v_2 \ra_f|}{(f(v_1)f(v_2))^{1/2}}.
\]
The image of a hyperplane $R \subset \Hy_f$ in a manifold $M = \Gamma \bs \Hy_f$ is a \emph{hypersurface} 
if the composite $R \inj \Hy_f \surj \Gamma \bs\Hy_f$ is an immersion.

Finally, if $M= \Gamma \bs \Hy_f$ is a manifold, its \emph{systole} is the length of the shortest geodesic 
in $M$.
It actually equals the \emph{minimal translation length} of $\Gamma$, defined as 
$\min\{d(x, \gamma x) \mid x \in \Hy_f, \gamma \in \Gamma$ hyperbolic$\}$.
Note that if $M$ is arithmetic of trace field $K$, there is a constant $\epsilon$ depending only on $K$ and 
the dimension $n$ such that the systole of $M$ is at least $\epsilon$ (see Remark~5.7 and the remark after 
Conjecture 10.4 in \cite{Gelander}). 

\subsection{}
The manifolds we will be interested in (we call them \emph{doubly-cut gluings})
arise from a specific gluing construction we now describe.
They were introduced by Belolipetsky and Thomson \cite{BT}, generalizing ideas of Agol \cite{Agol}.
We merely outline their construction here, and refer to \cite{Mila18} for a precise definition.

Let $\Gamma \subset \aPO_f(\calO_K)$ be an arithmetic lattice and let $R_1 = v_1^{\perp_f}$ and $R_2 = v_2^{\perp_f}$ be hyperplanes that are not incident (in $\Hy_f$).
Assume that both vectors $v_1, v_2$ are actually in $K^{n+1}$; we will say that $R_1$ and $R_2$ are 
\emph{rational} or \emph{$K$-rational}.
Then a \emph{doubly-cut gluing} constructed from this data is a manifold obtained using the following procedure:
\begin{enumerate}
    \item Select a finite-index subgroup $\Gamma_1 \subset \Gamma$ such that 
        %for any 
        %$\gamma, \gamma' \in \Gamma_1$ and for any $i,j \in \{1,2\}$, the translates $\gamma R_i$ and 
        %$\gamma' R_j$ are either disjoint or they coincide, in which case we require that $i = j$.
        the hyperplanes $R_1$ and $R_2$ project down to disjoint \emph{hypersurfaces} in 
        $M_1 = \Gamma_1 \bs \Hy_f$.
    \item Cut the manifold $M_1$ open at the two hypersurfaces and form its completion $M_2$; it is a manifold 
        with boundary.
    \item Form its \emph{double} $M_3$ by identifying the boundary components of two mirrored copies of $M_2$.
\end{enumerate}
The hyperplanes $R_1$ and $R_2$ are the \emph{cut hyperplanes}.
See \cite[Sect.~1.2]{Mila18} for a justification of the existence of the subgroup $\Gamma_1$ as well as other 
technical details.

If $R_1$ and $R_2$ are ultraparallel, the image in $M_3$ of the geodesic segment orthogonal to $R_1$ and 
$R_2$, together with its ``mirrored copy'', forms a closed geodesic in $M_3$ of length $2d(R_1, R_2)$.
Thus with hyperplanes getting closer and closer we get doubly-cut gluings with shorter and shorter systole.

%\pagebreak

%\vspace*{-2cm}

\section{The non-compact case} \label{sec:non-compact}
In this section we prove the following theorem.

\begin{theorem} \label{th:non-compact}
  For every $n \geq 3$, there exist infinitely many \emph{non-compact} pairwise non-commensurable hyperbolic 
  $n$-manifolds having the same ambient group and trace ring.
\end{theorem}
\begin{proof}
  Fix $n \geq 3$. 
  The manifolds we will construct are doubly-cut gluings with carefully chosen cut hyperplanes.
  Let $f = -x_0^2 + x_1^2 + \cdots + x_n^2$ be the standard signature $(n,1)$ quadratic 
  form, and $R_0 \subset \Hy_f$ the hyperplane defined by $R_0 = v^{\perp_f}$ with $v = (0,1,0, \dots, 0)$.
  For a rational hyperplane $R = w^{\perp_f}$, ultraparallel to $R_0$, we let $M_w$ denote a doubly cut gluing 
  constructed using the arithmetic manifold $\aPO_f(\Z)$ and the cut hyperplanes $R_0$ and $R$.

  The goal is to find a sequence $(w_k)_{k\geq 1} \subset \Z^{n+1}$ such that the following holds:
  \begin{enumerate}
      \item The squared norm $f(w_k) = \<w_k, w_k\>_f$ is positive, the hyperplanes $R_0$ and 
          $R_{k} = w_k^{\perp_f}$ are 
          ultraparallel and 
          \[
              d(R_0, R_k) \longrightarrow 0 \quad \text{as } k \to \infty.
          \]  
      \item The trace rings $A_k$ of the manifolds $M_{w_k} = \Gamma_{w_k} \bs \Hy_f$ 
          are all contained in a finitely generated subring $A$ of $\Q$.
  \end{enumerate}
  
  If we can find such a sequence, we can conclude using the following argument due to Agol~\cite{Agol} and 
  generalized by Belolipetsky and Thomson~\cite{BT}.
  First, observe that as the distance between $R_0$ and $R_k$ tends to zero, the systole of the manifolds 
  $M_{w_k}$ also tends to zero (see the explanations of the previous section), and thus they are 
  eventually all nonarithmetic.

  Assume there are only finitely many commensurability classes among the $M_{w_k}$.
  Up to passing to a subsequence, we can assume that all $M_{w_k}$ are commensurable and nonarithmetic.
  Hence the lattices $\Gamma_{w_k}$ are all contained in the commensurator $\Comm(\Gamma_{w_1})$ of, say, 
  $\Gamma_{w_1}$.
  Since $\Gamma_{w_1}$ is nonarithmetic, its commensurator is itself a lattice 
  (see \cite[Th.~1]{Margulis}), and thus has a positive minimal translation length.
  This contradicts the fact that the minimal translation length of $\Gamma_{w_k}$ tends to zero.
  
  Now up to finding a subsequence, we can assume that the $M_{w_k}$ are all pairwise non-commensurable.
  Since there are only a finite number of subrings of $A$, we can find a further subsequence such that 
  all elements $M_{w_k}$ have the same trace ring, completing the proof.

  We are left with finding such a sequence $(w_k)_{k\geq 1} \subset \Z^{n+1}$.
  First recall that, as explained above, 
  \[
      \left(\cosh d(R_0, R_k)\right)^2 = \frac{\<v,w_k\>_f^2}{f(v) f(w_k)} = \frac{w_{k,1}^2}{f(w_k)}.
  \]  
  Thus (a) is equivalent to 
  \begin{equation} \label{eq:1}
      w_{k,1}^2 > f(w_k) > 0 \quad \text{and} \quad \frac{w_{k,1}^2}{f(w_k)} \longrightarrow 1.
  \end{equation}
  (The first two inequalities mean that the hyperplanes are ultraparallel and that $w_k^{\perp_f}$ indeed 
  defines a hyperplane.)

  Using Lemma~2.6 of \cite{Mila18}, we get that the trace rings $A_k$ of the $M_{w_k}$ all satisfy:
  \[
      A_k \subset \Z\left[\frac{1}{f(w_k)}\right] = \Z\left[\frac{1}{p_1}, \dots, \frac{1}{p_r}
      \right],
  \]
  where $p_1, \dots, p_r$ are the prime factors of $f(w_k)$, and the equality follows from elementary 
  properties of subrings of $\Q$.
  Hence (b) can be achieved by finding a sequence $w_k$ such that all integers $f(w_k)$ are 
  $T$-smooth for some fixed constant $T$ (an integer is \emph{$T$-smooth}, or \emph{$T$-friable} if all its 
  prime factors are less than $T$).

  We are ready to construct our sequence $(w_k)_{k\geq 1}$.
  Let $T > 1$ be fixed and let $(r_k)_{k \geq 1}$ be any infinite sequence of positive $T$-smooth integers.
  For each $r_k$ we can find integers $b_k > 0$ and $0 \leq c_k \leq 2 b_k$ such that 
  \[
    r_k = b_k^2 - 2 c_k - 1, 
  \]  
  since all positive integers are contained in the union 
  \[
    \bigcup_{b \in \Z, b > 0} \{ b^2 - 2c -1 \mid c \in \Z, 0 \leq c \leq 2 b\}.
  \]  
  Defining
  \[
      w_k = (c_k + 1\,,\, b_k\,, \,c_k\,, 0, \dots, 0) \in \Z^{n+1},
  \]
  we see that 
  \[
      f(w_k) = -(c_k+1)^2 + b_k^2 + c_k^2 = b_k^2 -2 c_k - 1 = r_k
  \]  
  and hence all these square norms are $T$-smooth.

  Moreover, it is obvious that for our choices of $b_k, c_k$, we have
  \[
    b_k^2 > b_k^2 -2 c_k - 1 > 0,
  \]  
  and since 
  \[
    \frac{b_k^2}{b_k^2 - 2 c_k - 1} = \frac{b_k^2}{b_k^2 + o(b_k^2)} \longto 1,
  \]  
  the proof is complete.
\end{proof}
   
\begin{remark}
  If we let $r_k = p^k $ for some fixed prime $p$, we get that all our manifolds have trace ring 
  \emph{exactly} $\Z[\frac{1}{p}]$.
\end{remark}

\section{The compact case} \label{sec:compact}

This method can be adapted to work over number fields, thus producing compact manifolds with same 
ambient group and trace rings:

\begin{theorem} \label{th:compact}
  For every $n \geq 4$, there exist infinitely many \emph{compact} pairwise non-commensurable hyperbolic 
  $n$-manifolds having the same ambient group and trace ring.
\end{theorem}
\begin{proof}
    Let $K = \Q(\sqrt{5})$, and let $f$ denote the following quadratic form in $n+1 \geq 4$ variables:
    \[
        f = -\sqrt{5} \, x_0^2 + x_1^2 + \cdots + x_n^2.
    \]  
    As in the proof of the previous theorem, we will construct a sequence of doubly-cut gluings 
    $(M_{w_k})_{k \geq 1}$, this time using the $f$-hyperboloid model $\Hy_f$ for hyperbolic space.
    The cut hyperplanes for $M_{w_k}$ will be $R_0$ and $R_k$, where 
    $R_0 = \{x_1 = 0\} = v^{\perp_f} \subset \Hy_f$, with $v = (0,1,0, \dots, 0)$ and 
    $R_k= w_k^{\perp_f}$ is to be determined.
    
    Let $w = (\alpha,\beta, \gamma_1, \gamma_2, \gamma_3, 0, \dots,0) \in \calO_K^{n+1}$, 
    where $\calO_K = \Z\left[\frac{1 + \sqrt{5}}{2}\right]$ is the ring of integers of $K$.
    By Maass' Theorem (see \cite[Th.~14.3.2, p.~193]{Grosswald}), any totally positive element of $\calO_K$ 
    can be realized as a sum of three squares in $\calO_K$.
    Thus for any totally positive $\epsilon \in \calO_K$, we can find values for
    $\gamma_1, \gamma_2, \gamma_3$ 
    such that 
    \[
        f(w) = -\sqrt{5} \alpha^2 + \beta^2 + \epsilon
    \] 
    We will take $R_k = w_k^\perp$ for some $w_k$ of the above form.
    As before, our goal is to find $w_k$ such that the systole of $M_{w_k}$ goes to zero 
    while its trace ring 
    remains fixed.

    Let $\rho \in \Z[\sqrt{5}]$ be such that $\sigma(\rho)^2 > \rho > \sigma(\rho) > 1$, where 
    $\sigma$ is the non-trivial automorphism of $K$.
    (For instance, one can take $\rho = 6 + \sqrt{5}$).
    Write $\rho^k = u_k + \sqrt{5}\, v_k$ with $u_k, v_k \in \Z$, and define 
    \[
        \begin{split}
            \alpha_k &= \lceil x\rceil + \sqrt{5} \quad \text{and} \quad
            \beta_k = \left\lfloor \sqrt{5} y \right \rfloor + \sqrt{5} y 
        \quad \text{where}\\
            x &= \sqrt{\frac{\sigma(\rho^k)}{\sqrt{5}}}, \quad \text{and} \quad
            y = \left\lfloor \sqrt{\frac{u_k}{10}} \right \rfloor.
        \end{split}
    \]
    A tedious but straightforward calculation shows that 
    \[
        -\sqrt{5}\, \alpha_k^2 + \beta_k^2 = \rho^k - \epsilon
    \]
    where $\epsilon \in \Z[\sqrt{5}]$ is in $O(\beta_k)$ %$O(\sqrt{u_k})$ 
    and totally positive for $k$ large enough. 

    From the discussion above, it follows that there are $\gamma_1, \gamma_2, \gamma_3 \in \calO_K$ such that 
    if $w_k = (\alpha_k, \beta_k, \gamma_1, \gamma_2, \gamma_3, 0, \dots, 0)$ we have
    \[
      f(w_k) = \rho^k - \epsilon + \epsilon = \rho^k.
    \]  
    
    We are ready to conclude the proof.
    As previously, the conditions that $R_k$ is a hyperplane not intersecting $R_0$, and that the systole of 
    $M_{w_k}$ tends to zero can be fulfilled by choosing $w_k$ satisfying equation \eqref{eq:1} 
    in the proof of Theorem~\ref{th:non-compact}.
    We claim that this is the case for our $w_k$, for $k$ large enough.

    To prove that $\frac{w_{k,1}^2}{f(w_k)} \to 1$, observe first that
    \[
        \frac{\sqrt{5}\,v_k}{u_k} = \frac{\rho^k - \sigma(\rho^k)}{\rho^k + \sigma(\rho^k)}
        = \frac{1 - (\frac{\sigma(\rho)}{\rho})^k}{1 + (\frac{\sigma(\rho)}{\rho})^k} \longto 1.
    \]  
    From this we get that 
    \[
        \frac{f(w_k)}{w_{k,1}^2} = \frac{\beta_k^2 - \sqrt{5}\,\alpha_k^2 + O(\beta_k)}{\beta_k^2}
        = 1 - \frac{\sqrt{5}\,\alpha_k^2}{\beta_k^2} + o(1) \longto 1,
    \]  
    since 
    \[
        \frac{\alpha_k^2}{\beta_k^2} = \frac{1}{2\sqrt{5}}
        \frac{u_k - \sqrt{5}\,v_k + O(\sqrt{u_k})}{u_k + O(\sqrt{u_k})}
        = \frac{1}{2\sqrt{5}} \left(1 - \frac{\sqrt{5}\, v_k }{u_k+ o(u_k)}\right) \longto 0.
    \]
    Thus $\frac{w_{k,1}^2}{f(w_k)} \to 1$.

    It is clear that $f(w_k) = \rho^k >0$.
    To show that $w_{k,1}^2 > f(w_k)$, observe first that 
    \[
        \frac{\sigma(\rho)^{2k}}{u_k} \geq \frac{\sigma(\rho)^{2k}}{\rho^k} = 
        \left(\frac{\sigma(\rho)^2}{\rho} \right)^k \longto \infty,
    \]
    since $\sigma(\rho)^2 > \rho$ by hypothesis.
    It follows that 
    \[
        f(w_k) - w_{k,1}^2 = f(w_k) - \beta_k^2 = -\sqrt{5}\,\alpha_k^2 + O(\beta_k)
    \]  
    is negative for large enough $k$, since
    \[
        \frac{\alpha_k^4}{\beta_k^2} = \frac{1}{10} 
        \frac{\sigma(\rho^{2k}) + O(u_k^{3/2})}%
        {u_k + O(u_k^{1/2})} \longto +\infty.
    \]  

    We have proven that equation \eqref{eq:1} holds for $w_k$, for $k$ large enough.
    It now follows from Lemma 1.1 in \cite{Mila18} that the trace ring $A_k$ of $M_k$ satisfies 
    \[
        A_k \subset \calO_K\left[ \frac{1}{f(w_k)} \right] 
        = \calO_K\left[ \frac{1}{\rho^k} \right] 
        = \calO_K\left[ \frac{1}{\rho} \right].
    \]
    As there are only finitely many integrally closed subrings of 
    $\calO_K[\frac{1}{\rho}]$, we conclude as in the proof of Theorem~\ref{th:non-compact}.
\end{proof}
\begin{remark}
    The same proof idea works over any quadratic number field. 
    However, there is no analogue of Maass' Theorem here (in fact, Siegel showed that $\Q$ and $\Q(\sqrt{5})$
    are the only totally real fields whose totally positive integers are sums of squares of integers, 
    see \cite[Th.~14.3.3]{Grosswald}).

    To palliate this, one can consider instead quadratic forms of the shape 
    \[
        -\sqrt{d} x_0^2 + x_1^2 
        + (\big\lceil\sqrt{d}\big\rceil - \sqrt{d})(x_2^2 + \cdots + x_5^2) + x_6^2 + \cdots + x_n^2.
    \]
    Here for $n \geq 9$, the above error $\epsilon$ can be compensated by using the 
    Four Squares Theorem twice, once with 
    variables $x_2, \dots, x_5$ and once with $x_6, \dots, x_9$.
    From there, essentially the same method works to produce compact examples.
    
    Also, it is to be noted that there is nothing special about the case $n = 3$ in our context, 
    and with more precise calculations one could probably include it in the compact version.
\end{remark}

{\small
%\nocite{*}
\bibliographystyle{abbrv}
\bibliography{bibliography} 
}

\end{document}